\documentclass[11pt]{amsart}
\usepackage[utf8]{inputenc}
\usepackage[T1]{fontenc}
\usepackage{amsmath,amssymb,amsfonts,textcomp,amsthm,xifthen,graphicx,color,pgfplots}
\usepackage{enumerate}
\usepackage{enumitem}
\usepackage{pgfplotstable}
\usepgfplotslibrary{groupplots}
\usepgfplotslibrary{colorbrewer}
\pgfplotsset{compat=newest}
\usepgfplotslibrary{external}
\usepackage{fullpage}
\usepackage[pdftex,
  pdfauthor={Thomas F\"{u}hrer and Gregor Gantner},
            pdftitle={Aubin--Nitsche-type estimates for space-time FOSLS for parabolic PDEs},
            ]{hyperref}

\newtheorem{theorem}{Theorem}

\newtheorem{lemma}[theorem]{Lemma}

\newtheorem{proposition}[theorem]{Proposition}
\newtheorem{remark}[theorem]{Remark}

\newcommand{\patch}{\omega}

\def\di{\mathrm{d}}

\DeclareMathOperator*{\argmin}{arg\, min}

\newcommand{\proj}{\Pi}
\newcommand{\projHmOne}{\mathcal{Q}}

\newcommand{\ip}[2]{(#1\hspace*{.5mm},#2)}
\newcommand{\dual}[2]{\langle#1\hspace*{.5mm},#2\rangle}

\newcommand{\diam}{\mathrm{diam}}

\renewcommand{\div}{\operatorname{div}}

\newcommand{\divx}{\operatorname{div}_{\bx}}
\newcommand{\divst}{\operatorname{div}_{t,\bx}}
\newcommand{\gradx}{\nabla_{\bx}}

\newcommand{\Hdivset}[1]{\boldsymbol{H}(\divx;#1)}
\newcommand{\Hdivstset}[1]{\boldsymbol{H}(\divst;#1)}

\newcommand{\set}[2]{\big\{#1\,:\,#2\big\}}

\newcommand{\RT}{\ensuremath{\boldsymbol{RT}}}

\newcommand{\R}{\ensuremath{\mathbb{R}}}
\newcommand{\N}{\ensuremath{\mathbb{N}}}

\newcommand{\cT}{\ensuremath{\mathcal{T}}}

\newcommand{\cK}{\ensuremath{\mathcal{T}}}

\newcommand{\ran}{\operatorname{ran}}

\newcommand\grad{\nabla}

\newcommand{\bsigma}{{\boldsymbol\sigma}}

\newcommand{\btau}{{\boldsymbol\tau}}
\newcommand{\bchi}{{\boldsymbol\chi}}

\newcommand{\bx}{{\boldsymbol{x}}}

\newcommand{\bu}{\boldsymbol{u}}
\newcommand{\bv}{\boldsymbol{v}}
\newcommand{\bw}{\boldsymbol{w}}

\newcommand{\iop}{\mathcal{I}}

\begin{document}

\title{Aubin--Nitsche-type estimates for space-time FOSLS \\ for parabolic PDEs}
\date{\today}

\author{Thomas F\"uhrer}
\address{Facultad de Matem\'{a}ticas, Pontificia Universidad Cat\'{o}lica de Chile, Santiago, Chile}
\email{thfuhrer@uc.cl}

\author{Gregor Gantner}
\address{Institute for Numerical Simulation, University of Bonn, Bonn, Germany}
\email{gantner@ins.uni-bonn.de}

\thanks{{\bf Acknowledgment.} 
This work was supported by ANID through FONDECYT project 1210391 (TF) and the Deutsche Forschungsgemeinschaft (DFG, German Research Foundation)
under Germany's Excellence Strategy – EXC-2047/1 – 390685813 (GG)}


\begin{abstract}
  We develop Aubin--Nitsche-type estimates for recently proposed first-order system least-squares finite element methods (FOSLS) for the heat equation.
  Under certain assumptions, which are satisfied if the spatial domain is convex and the heat source and initial datum are sufficiently smooth, we prove that the $L^2$ error of approximations of the scalar field variable converges at a higher rate than the overall error. 
  Furthermore, a higher-order conservation property is shown.
  In addition, we discuss quasi-optimality in weaker norms.
  Numerical experiments confirm our theoretical findings.
\end{abstract}
\maketitle

\section{Introduction}
In this article, we consider approximations to the minimization problem
\begin{align}\label{eq:minprob}
  \bu = (u,\bsigma) = \argmin_{\bv=(v,\btau)\in U} \|\divst\bv-f\|_{L^2(Q)}^2 + 
  \|\gradx v+\btau\|_{L^2(Q)}^2 + \|v(0)-u_0\|_{L^2(\Omega)}^2, 
\end{align}
which is equivalent to the heat equation, cf.~\cite{FuehrerKarkulik21,GantnerStevenson21},
\begin{subequations}\label{eq:heat}
\begin{alignat}{2}
  \partial_t u -\Delta_{\bx} u  &= f &\quad&\text{in } Q:=I\times\Omega, \\
  u(0)  &= u_0 &\quad&\text{in } \Omega, \\
  u  &= 0 &\quad&\text{on } I\times \partial\Omega,
\end{alignat}
\end{subequations}
and $\bsigma = -\gradx u$.
Here, $\Omega\subseteq \R^d$ ($d=1,2,3$) denotes a bounded Lipschitz domain, $I=(0,T)$ a finite time interval with (fixed) end time $T>0$, $f\in L^2(\Omega)$ the volumetric heat source, and $u_0\in L^2(\Omega)$ the initial temperature distribution.
Further, $\divst(v,\btau) = \partial_t v + \divx\btau$ denotes the space-time divergence and $U$ is a Hilbert space given by
\begin{align*}
  U := \Hdivstset{Q} \cap \big(L^2( I;H_0^1(\Omega))\times L^2(Q)^d\big).
\end{align*}
Approximations are obtained through least-squares finite element methods, which consist in replacing the space $U$ in~\eqref{eq:minprob} by a discrete subspace $U_h\subseteq U$ of finite element functions. 
This yields the first-order system least-squares method (FOSLS) 
\begin{align}\label{eq:fosls}
\bu_h = (u_h,\bsigma_h) = \argmin_{\bv=(v,\btau)\in U_h} \|\divst\bv-f\|_{L^2(Q)}^2 + 
\|\gradx v+\btau\|_{L^2(Q)}^2 + \|v(0)-u_0\|_{L^2(\Omega)}^2.
\end{align}
In~\cite{FuehrerKarkulik21,GantnerStevenson21} it was shown that~\eqref{eq:fosls} admits a unique solution and is quasi-optimal with respect to the (graph) norm on $U$.

The motivation of this work is to develop Aubin--Nitsche-type duality arguments to prove (in Theorem~\ref{thm:L2est}) under certain regularity assumptions that
\begin{align}\label{eq:L2est:general}
  \|u-u_h\|_{L^2(Q)} \le \gamma(h)\|\bu-\bu_h\|_U,
\end{align}
where $0\leq \gamma(h)\to 0$ for $h\to 0$ and $h$ denotes a discretization parameter (mesh-size).
Thus, the $L^2(Q)$ error converges faster than the error in the whole $U$ norm.
With similar techniques, we show our second main result (Theorem~\ref{thm:conservation}), 
\begin{align}\label{eq:conservation}
  \|\proj f-\divst\bu_h\|_{L^2(Q)} \le \gamma(h)\|\bu-\bu_h\|_U,
\end{align}
where $\proj\colon L^2(Q)\to \ran(\divst|_{U_h})$ denotes the orthogonal projection.
For elliptic model problems similar estimates can be interpreted as higher-order mass conservation, see, e.g.,~\cite{BrandtsChenYang06}.

While the Aubin--Nitsche trick is a well-established tool for variational methods for elliptic problems, 
the situation is different in the parabolic case.
For elliptic problems, $L^2$ norm error estimates for FOSLS have been studied in~\cite{CaiKu2006,Ku11,MINRESsingular,MR4532778,MR4784402}.
These works, as is usual for techniques based on duality arguments, rely on elliptic regularity results. 
For parabolic problems though, regularity results are of anisotropic nature (derivatives with respect to time have different regularities than spatial derivatives). 
This has to be taken into account in the analysis by considering certain intermediate spaces.
Our results also rely on recent findings on approximation properties and boundedness of interpolation operators~\cite{StevensonStorn23} for the finite element space introduced in~\cite{GantnerStevenson24}.

Building on stability properties of these interpolation operators, and as in \cite{FuehrerMixedFEMFOSLS} for the stationary case, we finally note a quasi-optimality without the divergence term in the $U$-norm, i.e.,
\begin{align}\label{eq:quasioptX}
    \|u-u_h\|_X + \|\bsigma-\bsigma_h\|_{L^2(Q)} \lesssim \min_{(v_h,\btau_h)\in U_h}
    \|u-v_h\|_X + \|\bsigma-\btau_h\|_{L^2(Q)},
  \end{align}
if $f$ is replaced for the computation of $u_h$ by $\projHmOne f$ with some suitable projection $\projHmOne$.
This observation has already been briefly noted in~\cite[Remark~19]{StevensonStorn23} to some extent, and, here, we provide some details of the proof.

\subsection{Outline}
The remainder of this work is organized as follows: 
In Section~\ref{sec:bochner-sobolev}, we recalll Bochner--Sobolev spaces as well as unique solvability of the heat equation. 
In Section~\ref{sec:regularity}, we further recall a standard parabolic regularity result. 
Section~\ref{sec:fosls} introduces the variational formulation of~\eqref{eq:minprob} along with a discrete version.
Next, in Section~\ref{sec:meshes},  we define regular tensor-product meshes and recall the  trial space of~\cite{GantnerStevenson24}.
In Section~\ref{sec:iop}, we state the employed properties of the interpolation operators of \cite{StevensonStorn23}.
We present our first two main results~\eqref{eq:L2est:general} and~\eqref{eq:conservation} in Section~\ref{sec:l2_estimates}.
Section~\ref{sec:observations} gives at least some partial explanation why higher rates can also be expected for the $L^2$ error at final time $T$.
In Section~\ref{sec:quasi_optimality}, we finally exploit stability properties of the interpolation operators to conclude the quasi-optimality~\eqref{eq:quasioptX} without the divergence term. 
The paper is concluded by numerical experiments in Section~\ref{sec:observations}, confirming our theoretical findings.

\subsection{General notation}
We write $A\lesssim B$ to abbreviate $A\le CB$ with some generic constant $C>0$ which is clear from the context, and which might only depend on the considered spatial domain $\Omega$, the end time $T$, and quasi-uniformity~\eqref{eq:quasi-uniform} and shape regularity~\eqref{eq:shape-regular} of the considered space-time mesh. 
Moreover, $A\eqsim B$ abbreviates $A\lesssim B\lesssim A$.

\section{Preliminaries}\label{sec:analysis}

\subsection{Bochner--Sobolev spaces}\label{sec:bochner-sobolev}
For Lipschitz domains $\omega\subset \R^d$ (or $\omega\subset\R^{d+1}$), we denote by $L^2(\omega)$ the Lebesgue space of square-integrable functions equipped with norm $\|\cdot\|_{L^2(\omega)}$ and inner product $\ip\cdot\cdot_{L^2(\omega)}$. 
With $H^k(\omega)$, we denote Sobolev spaces of order $k\in\N$ and $H_0^1(\omega)$ is the closed subspace of $H^1(\omega)$ with vanishing trace.
We equip $H_0^1(\omega)$ with norm $\|\gradx(\cdot)\|_{L^2(\omega)}$.
The dual space of $H_0^1(\omega)$ is $H^{-1}(\omega)$, where duality is understood with respect to the extended $L^2(\omega)$ inner product.

For strongly measurable $v\colon I\to H$, where $H$ is a Banach space with norm $\|\cdot\|_H$, we define
\begin{align*}
  \|v\|_{L^2( I;H)}^2 := \int_0^T \|v(t)\|_H^2 \,\di t
\end{align*}
and denote by $L^2( I; H)$ the Bochner space of all such $v$ with $\|v\|_{L^2( I;H)}<\infty$.
For an introduction and basic results on Bochner spaces, we refer the interested reader, e.g., to \cite{DL92}.
We note that $L^2( I; L^2(\Omega))$ is isometrically isomorphic to $L^2(Q)$, and we identify these two spaces in the remainder of this article.
With $H^1( I;H)$ we denote the space of all $v\in L^2( I;H)$ such that $\partial_t v\in L^2( I;H)$.
Note the continuous embedding $H^1( I;H) \hookrightarrow C^0(\overline I;H)$, see, e.g., \cite[Sec.~5.9, Theorem~2]{Evans98}, where the latter space denotes the space of functions $v\colon \overline I\to H$ which are continuous in time with values in $H$.
For $s\in (0,1)$, we define the intermediate spaces $H^s( I;H)$ by interpolation between $L^2( I; H)$ and $H^1( I; H)$.
We use similar notation for open subintervals $J\subseteq I$ of $I$.

The natural trial space for the weak formulation of the heat equation is given by
\begin{align*}
  X := L^2( I; H_0^1(\Omega)) \cap H^1( I;H^{-1}(\Omega)).
\end{align*}
We first recall the continuous embedding 
\begin{align}\label{eq:X2C0}
  X \hookrightarrow C^0(\overline I;L^2(\Omega))
\end{align}
from, e.g., \cite[Sec.~5.9, Theorem~3]{Evans98}, and then  the following theorem from, e.g., \cite[Theorem~5.1]{ss09}.

\begin{theorem}\label{thm:ss09}
  For given $f \in L^2(I;H^{-1}(\Omega))$ and $u_0\in L^2(\Omega)$, there exists a unique weak solution $u\in X$ of the heat equation~\eqref{eq:heat}. 
  It satisfies the stability estimate
  \begin{align}\label{eq:X2RHS}
    \| u \|_X \lesssim \| f \|_{L^2(I;H^{-1}(\Omega))} + \| u_0 \|_{L^2(\Omega)}. 
  \end{align}
  \qed
\end{theorem}

Recall that 
\begin{align*}
  U := \Hdivstset{Q} \cap \big(L^2( I;H_0^1(\Omega))\times L^2(Q)^d\big),
\end{align*}
which we equip with the squared norm
\begin{align*}
  \|\bv\|_U^2 = \|\gradx v\|_{L^2(Q)}^2 + \|\btau\|_{L^2(Q)}^2 + \|\divst\bv\|_{L^2(Q)}^2 
  \quad\forall \bv=(v,\btau)\in U.
\end{align*}
According to \cite[Lemma~2.2]{GantnerStevenson21}, it holds for any $\bv=(v,\btau)\in U$ that $v\in X$ with
\begin{align}\label{eq:X2U}
  \| v \|_X \lesssim \| \bv \|_U.
\end{align}
As already mentioned in the introduction, \eqref{eq:minprob} and \eqref{eq:heat} are equivalent, cf.\ \cite{FuehrerKarkulik21,GantnerStevenson21}. 

\begin{theorem}
Let $f\in L^2(Q)$ and $u_0\in L^2(\Omega)$. 
Then, the unique solution $u\in X$ of the heat equation~\eqref{eq:heat} satisfies that $\bu:=(u, -\gradx u)\in U$ and $\bu$ solves minimization problem~\eqref{eq:minprob}.
Conversely, if $\bu=(u, \bsigma)\in U$ is a solution of~\eqref{eq:minprob}, then $u\in X$ and $u$ solves~\eqref{eq:heat} with $\bsigma = -\gradx u$. 
\qed
\end{theorem}

\subsection{Parabolic regularity}\label{sec:regularity}

The next result states that the weak solution $u\in X$ of the heat equation is even in the space
\begin{align*}
  Y := L^2( I; H^2(\Omega)) \cap C^0( \overline I; H_0^1(\Omega)) \cap H^1( I; L^2(\Omega))
\end{align*}
under additional regularity assumptions.
The proof follows the lines of argumentation of~\cite[Sec.~7.1, Theorem~5]{Evans98}.
We note that in~\cite{Evans98}, $\Omega$ is assumed to have smooth boundary. This is only required to deduce elliptic regularity which can also be shown by considering convex Lipschitz domains as is done below. 
Another difference is that~\cite[Sec.~7.1, Theorem~5]{Evans98} states the improved regularity $u\in L^\infty( I;H_0^1(\Omega))$, while actually the slightly stronger statement $u \in C^0(\overline I;H_0^1(\Omega))$ holds true (this is a simple consequence of, e.g.,~\cite[Ch.~XVIII, Eq.(1.61)]{DL92} and the fact that the intermediate interpolation space between $H^2(\Omega)$ and $L^2(\Omega)$ is $H^1(\Omega)$ for a Lipschitz domain $\Omega$). 

\begin{proposition}[Parabolic regularity]\label{prop:parabolicreg}
  Suppose that $\Omega$ is convex.
  Let $u\in X$ be the unique weak solution of~\eqref{eq:heat} with $f\in L^2(Q)$ and $u_0\in H_0^1(\Omega)$. 
  Then, $u\in Y$ with
  \begin{align*}
    \|u\|_{Y} \lesssim \|f\|_{L^2(Q)} + \|\gradx u_0\|_{L^2(\Omega)}.
  \end{align*}
  \qed
\end{proposition}

\begin{remark}
  Note that \cite{Evans98} considers general second-order parabolic PDEs with smooth coefficients. 
  Again, this result readily extends to Lipschitz domains $\Omega$. 
  As this is the crucial ingredient for our Aubin--Nitsche-type estimates (of Theorem~\ref{thm:L2est} and Theorem \ref{thm:conservation}), we stress that the analysis of the present work holds equally well true for general second-order parabolic PDEs with smooth coefficients;
  see also~\cite{GantnerStevenson21} for least-squares formulations in this case. 
  \qed
\end{remark}

\begin{remark}\label{rem:parabolicreg}
  The statements of Proposition~\ref{prop:parabolicreg} hold true with obvious modifications if we consider the (weak formulation of the) backward heat equation
  \begin{align*}
    -\partial_t u - \Delta_\bx u &= f, \\
    u(T) &= u_T, \\
    u|_{I\times\partial\Omega} &= 0
  \end{align*}
  with $f\in L^2(Q)$ and $u_T\in H_0^1(\Omega)$.
\end{remark}

We will also employ the following result.
\begin{lemma}\label{lem:embedding}
  The embedding $Y\hookrightarrow  H^{1/2}( I;H^1(\Omega))$ is continuous. 
  \qed
\end{lemma}
\begin{proof}
 We recall that $[H^2(\Omega),L^2(\Omega)]_{1/2}= H^1(\Omega)$ with equivalent norms, where $[A,B]_{s}$, $s\in(0,1)$ denotes the interpolation space of a compatible pair of Hilbert spaces $A,B$.
    By~\cite[Ch.1, Remark~9.5]{LionsMagenesI}, it follows that
    \begin{align*}
      [L^2(I;H^2(\Omega)),H^1(I;L^2(\Omega))]_{1/2} = H^{1/2}(I;H^1(\Omega)).
    \end{align*}
    Therefore, $Y$ is included in $H^{1/2}(I;H^1(\Omega))$. Continuity of the inclusion follows from the standard interpolation estimate
    \begin{align*}
      \|u\|_{H^{1/2}(I;H^1(\Omega))} \leq \|u\|_{L^2(I;H^2(\Omega))}^{1/2}\|u\|_{H^1(I;L^2(\Omega))}^{1/2}
      \leq \|u\|_Y,
  \end{align*}
which finishes the proof.
\end{proof}

\subsection{Variational formulation of FOSLS}\label{sec:fosls}
Define the bilinear form $a\colon U\times U \to \R$ for $\bu=(u,\bsigma), \bv=(v,\btau)\in U$ by
\begin{align*}
  a(\bu,\bv) := \ip{\divst\bu}{\divst\bv}_{L^2(Q)} + \ip{\gradx u+\bsigma}{\gradx v+\btau}_{L^2(Q)} + \ip{u(0)}{v(0)}_{L^2(\Omega)}.
\end{align*}
Further, given $f\in L^2(Q)$, $u_0\in L^2(\Omega)$, we introduce the linear functional $F\colon U\to \R$ for $\bv =(v,\btau)\in U$ by
\begin{align*}
  F(\bv) := \ip{f}{\divst\bv}_{L^2(Q)} + \ip{u_0}{v(0)}_{L^2(\Omega)}.
\end{align*}
The bilinear form $a(\cdot,\cdot)$ as well as the functional $F(\cdot)$ are bounded, i.e.,
\begin{align*}
  |a(\bu,\bv)|\lesssim \|\bu\|_U\|\bv\|_U, \qquad |F(\bv)| \lesssim (\|f\|_{L^2(Q)} + \|u_0\|_{L^2(\Omega)})\|\bv\|_U
  \qquad \forall\bu,\bv\in U.
\end{align*}

The variational formulation (Euler--Lagrange equations) of minimization problem~\eqref{eq:minprob} reads:
\begin{align}\label{eq:EL}
  \text{Find }\bu\in U: \qquad
  a(\bu,\bv) = F(\bv) \quad\forall \bv\in U.
\end{align}
Similarily, the variational formulation of problem~\eqref{eq:fosls} reads:
\begin{align}\label{eq:EL:disc}
  \text{Find } \bu_h\in U_h: \qquad
  a(\bu_h,\bv_h) = F(\bv_h) \quad\forall \bv_h\in U_h.
\end{align}
Solutions to~\eqref{eq:EL} and~\eqref{eq:EL:disc} satisfy the Galerkin orthogonality,
\begin{align}\label{eq:galerkinorth}
  a(\bu-\bu_h,\bv_h) = 0 \quad\forall \bv_h\in U_h
\end{align}
which is an important ingredient in the proofs of our main results. 

Formulations~\eqref{eq:EL} and~\eqref{eq:EL:disc} have been studied thouroughly in~\cite{FuehrerKarkulik21,GantnerStevenson21}. 
We recall some important observations of these works.
\begin{proposition}\label{prop:fosls}
  Let $f\in L^2(Q)$ and $u_0\in L^2(\Omega)$ be given. 
  Minimization problems~\eqref{eq:minprob} and~\eqref{eq:fosls} are equivalent to formulations~\eqref{eq:EL} and~\eqref{eq:EL:disc}, respectively. 
  Moreover, the bilinear form $a(\cdot,\cdot)$ is coercive, i.e., 
  \begin{align}\label{eq:coercivity}
  \| \bu \|_U^2 \lesssim a(\bu,\bu) \quad \forall \bu \in U.
  \end{align}
  In particular,  \eqref{eq:EL} and~\eqref{eq:EL:disc} (and the equivalent \eqref{eq:minprob} and~\eqref{eq:fosls}) admit unique solutions $\bu\in U$, $\bu_h\in U_h$, where $\bu_h$ is quasi-optimal, i.e., 
  \begin{align}\label{eq:quasiopt}
    \|\bu-\bu_h\|_U \lesssim \min_{\bv_h\in U_h} \|\bu-\bv_h\|_U.
  \end{align}
  \qed
\end{proposition}

\subsection{Meshes and discretization space}\label{sec:meshes}

Let $\cK_t$ denote a partition into open intervals of the time interval $I$ and let $\cK_\bx$ denote a partition of $\Omega$ into open simplices. 
The tensor-product mesh is defined as $\cK = \cK_t\otimes \cK_\bx$, i.e., for each $K\in\cK$ there exist unique $K_t\in\cK_t$, $K_\bx\in \cK_\bx$ with $K = K_t\times K_\bx$.
We define the maximal mesh sizes in time and space as
\begin{align*}
  h_t := \max_{K_t\in\cK_t} |K_t| \quad \text{and} \quad h_\bx := \max_{K_\bx\in\cK_\bx} \diam(K), 
\end{align*}
respectively.
Here, $|\cdot|$ denotes the length of an interval and $\diam(\cdot)$ the diameter of a set.
Throughout, we suppose that the meshes $\cK_t$ and $\cK_\bx$ are quasi-uniform in the sense that there exists a uniform constant $C>0$ independent of $\cK_t$ and $\cK_\bx$ such that
\begin{align}\label{eq:quasi-uniform}
  h_t \le C|K_t| \quad \forall K_t\in\cK_t \qquad \text{and} \qquad h_\bx \le C\diam(K_\bx) \quad \forall K_\bx\in\cK_\bx.
\end{align}
Moreover, we suppose that $\cK_\bx$ is shape regular in the sense that there exists a uniform constant $C'>0$ independent of $\cK_\bx$ such that
\begin{align}\label{eq:shape-regular}
  \diam(K_\bx)^d \le C' |K_\bx| \quad \forall K_\bx \in \cK_\bx
\end{align}
Here, $|\cdot|$ denotes the measure of a set. 
Note that this is trivially satisfied with $C'=1$ if $d=1$.
Given $K_t\in\cK_t$ and $K_\bx\in\cK_\bx$, we define the corresponding patches as 
\begin{align*}
  \patch(K_t) := {\rm int}\Big(\bigcup \set{\overline{K_t'}}{K_t'\in\cK_t, \overline{K_t} \cap \overline{K_t'} \neq \emptyset}\Big)
\end{align*}
and
\begin{align*}
  \patch(K_\bx) := {\rm int} \Big(\bigcup \set{\overline{K_\bx'}}{K_\bx'\in\cK_\bx, \overline{K_\bx} \cap \overline{K_\bx'} \neq \emptyset}\Big).
\end{align*}
Here, ${\rm int}(\cdot)$ denotes the interior of a set.

Let $P_k(\cK_t)$ and $P_k(\cK_\bx)$ denote the space of piecewise polynomials of degree $\leq k\in \N_0$ on $\overline I$ and $\overline \Omega$, respectively. 
Further, let $S_{k}(\cK_t)$ and $S_{k}(\cK_\bx)$ denote the space of continuous piecewise polynomials (Lagrange finite element space), and $S_{k,0}(\cK_\bx)$ the subspace of functions in $S_{k}(\cK_\bx)$ with vanishing trace on $\partial\Omega$. 
Moreover, let $\RT_k(\cK_\bx)$ be the $H(\div_\bx;\Omega)$-conforming Raviart--Thomas finite element space of order $k$ over the spatial partition $\cK_\bx$.
If $d=1$, $\RT_k(\cK_\bx)$ just coincides with $S_{k+1}(\cK_\bx)$.
For $K_t\in\cK_t$ and $H$ a Hilbert space, we use the notation $P_k(K_t;H)$ for $H$-valued polynomials of degree $\leq k$ in time.
Similarly, we define $P_k(\cK_t;H)$ as the space of functions on $\overline I$ whose restriction to each $K_t\in\cK_t$ is in $P_k(K_t;H)$, and $S_k(\cK_t;H)$ as the subspace of functions in $P_k(\cK_t;H)$ which are continuous in time.

Throughout this article, we consider, for arbitrary but fixed polynomial degrees $k,\ell\in \N$,
\begin{align}\label{eq:U_h}
  U_h := S_{k}(\cK_t)\otimes S_{\ell,0}(\cK_\bx) \times P_{k-1}(\cK_t)\otimes \RT_\ell(\cK_\bx),
\end{align}
where $\otimes$ denotes the algebraic tensor product of two finite-dimensional spaces.
This space has been introduced and studied in~\cite{GantnerStevenson24} for FOSLS for parabolic PDEs.

\subsection{Interpolation operators}\label{sec:iop}

Let $\proj\colon L^2(Q) \to P_{k-1}(\cK_t)\otimes P_\ell(\cK_\bx) = \divst(U_h)$ be the $L^2(Q)$ orthogonal projection onto $P_{k-1}(\cK_t)\otimes P_\ell(\cK_\bx)$.
The following simple standard approximation result will be used in the proofs of the main results. 
\begin{lemma}\label{lem:L2proj}
  For all $v\in L^2( I;H^2(\Omega))\cap H^1( I; L^2(\Omega))$, it holds that
  \begin{align*}
    \|(1-\proj)v\|_{L^2(Q)} \leq \min_{w\in P_{0}(\cK_t)\otimes P_1(\cK_\bx)}\|v-w\|_{L^2(Q)} \lesssim h_\bx^2\|D_\bx^2\|_{L^2(Q)} + h_t\|\partial_t v\|_{L^2(Q)}.
  \end{align*}
  \qed
\end{lemma}

For $H$ a Hilbert space and $k\in\N_0$, we denote by $\proj_{t,k}$ the $L^2(I;H)$ orthogonal projection onto $P_k(\cT_t;H)$.
We note that $\|(1-\proj_{t,k})v\|_{L^2(I;H)} \leq \|(1-\proj_{t,0})v\|_{L^2(I;H)}$ for all $v\in L^2(I;H)$.

\begin{lemma}\label{lem:projP0time}
  For all $s\in[0,1]$ and $v\in H^s( I; H)$, it holds that
  \begin{align*}
    \|(1-\proj_{t,0})v\|_{L^2(I;H)} \lesssim h_t^s \|v\|_{H^s( I;H)}.
  \end{align*}
  If $s=1$, it even holds that
    \begin{align*}
    \|(1-\proj_{t,0})v\|_{L^2( I;H)} \lesssim h_t \|\partial_tv\|_{L^2( I;H)}.
  \end{align*}
  For all $\btau\in L^2(Q)^d$, it holds that
  \begin{align*}
    \|(1-\proj_{t,0})\divx\btau\|_{L^2( I;H^{-1}(\Omega))} 
    \lesssim \|(1-\proj_{t,0})\btau\|_{L^2(Q)}.
  \end{align*}
\end{lemma}
\begin{proof}
  The first estimate for $s=0,1$ and the second (stronger) estimate are standard and follow as for the special case $H=\R$.
  The cases $s\in(0,1)$ are a consequence of standard interpolation arguments.

  To see the final estimate, note that $\proj_{t,0} \circ \divx = \divx\circ\proj_{t,0}$. Therefore,
  \begin{align*}
    \|(1-\proj_{t,0})\divx\btau\|_{L^2( I;H^{-1}(\Omega))} = \|\divx(1-\proj_{t,0})\btau\|_{L^2( I;H^{-1}(\Omega))}.
  \end{align*}
  Boundedness of $\divx\colon L^2(Q)^d \to L^2( I;H^{-1}(\Omega))$ finishes the proof.
\end{proof}

The following result stems from~\cite[Theorem~12 and Remark~14]{StevensonStorn23}.
Note that we consider the operator described in~\cite[Remark~14]{StevensonStorn23} to avoid necessity of parabolic scaling. 
\begin{lemma}\label{lem:iop1}
  There exists a linear projection $\iop_1\colon X\to S_{k}(\cK_t)\otimes S_{\ell,0}(\cK_\bx)$ such that for all $v\in X$, 
  \begin{align}
    \|\gradx\iop_1 v\|_{L^2(Q)} &\lesssim \|\gradx v\|_{L^2(Q)}, \label{eq:grad_iop1} \\
    \|\partial_t\iop_1 v\|_{L^2( I;H^{-1}(\Omega))} &\lesssim \|\partial_t v\|_{L^2( I;H^{-1}(\Omega))}. \label{eq:dt_iop1}
  \end{align}
  In addition, if $v\in Y$, then
  \begin{align}\label{eq:app_grad_iop1}
    \|\gradx(v-\iop_1v)\|_{L^2(Q)} &\lesssim (h_\bx + h_t^{1/2})\|v\|_Y,
  \end{align}
  If $v\in X\cap H^1( I;L^2(\Omega))$, then
  \begin{align}\label{eq:app_dt_iop1}
    \|\partial_t(v-\iop_1 v)\|_{L^2( I;H^{-1}(\Omega))} &\lesssim h_\bx \|\partial_t v\|_{L^2(Q)} + \|(1-\proj_{t,k-1})\partial_t v\|_{L^2( I;H^{-1}(\Omega))}.
  \end{align}
\end{lemma}
\begin{proof}
  The interpolation operator $\iop_1$ of~\cite[Remark~14]{StevensonStorn23} satisfies that
  \begin{align*}
    \|\gradx(v-\iop_1v)\|_{L^2(Q)}^2 &\lesssim \sum_{K\in\cK} \Big( \min_{v_\bx\in L^2(I;S_{\ell,0}(\cK_\bx))} \|\gradx(v-v_\bx)\|_{L^2(K_t;L^2(\patch(K_\bx)))}^2 \\
    &\qquad + \min_{\btau_t\in S_k(\cK_t;L^2(\Omega)^{d})} \|\gradx v - \btau_t\|_{L^2(\patch(K_t);L^2(\patch(K_\bx)))}^2\Big),  \\
    \|\partial_t(v-\iop_1 v)\|_{L^2( I;H^{-1}(\Omega))}^2 &\lesssim
    \sum_{K\in\cK}  \min_{v_\bx\in L^2(I;S_{\ell,0}(\cK_\bx))} \|\partial_tv-v_\bx\|_{L^2(K_t;H^{-1}(\patch(K_\bx)))}^2 
      \\
      &\qquad + \sum_{K_t\in\cK_t} \min_{v_t\in S_k(\cK_t;H^{-1}(\Omega))} \|\partial_t(v-v_t)\|_{L^2(\patch(K_t);H^{-1}(\Omega))}^2. 
  \end{align*}
  The boundedness estimates follow at once.

  Suppose that $v\in Y$. Approximation properties and noting $\gradx v\in H^{1/2}( I;L^2(\Omega)^d)$ 
  by the embedding $Y\hookrightarrow H^{1/2}(I;H^1(\Omega))$ yield that
  \begin{align*}
    \sum_{K\in\cK} \min_{v_\bx\in L^2(I;S_{\ell,0}(\cK_\bx))} \|\gradx(v-v_\bx)\|_{L^2(K_t;L^2(\patch(K_\bx)))}^2 &\lesssim h_\bx^2 \|D_\bx^2 v\|_{L^2(Q)}^2, \\
    \sum_{K\in\cK}\min_{\btau_t\in S_k(\cK_t;L^2(\Omega)^{d})} \|\gradx v - \btau_t\|_{L^2(\patch(K_t);L^2(\patch(K_\bx)))}^2
    &\lesssim h_t \|\gradx v\|_{H^{1/2}( I;L^2(\Omega)^d)}^2.
  \end{align*}
  Note that $\|D_\bx^2 v\|_{L^2(Q)} + \|\gradx v\|_{H^{1/2}( I;L^2(\Omega)^d)} \lesssim \|v\|_Y$.
  The very last estimate for $v\in X\cap H^1( I;L^2(\Omega))$ can be seen as follows: First, a scaling argument shows that 
  \begin{align*}
    \min_{v_\bx\in L^2(I;S_{\ell,0}(\cK_\bx))} \|\partial_tv-v_\bx\|_{L^2(K_t;H^{-1}(\patch(K_\bx)))} 
    \lesssim h_\bx \|\partial_t v\|_{L^2(K_t;L^2(\patch(K_\bx)))}.
  \end{align*}
  Second, by choosing $v_t:= \int \proj_{t,k-1}\partial_t v \,\di s\in S_{k}(\cK_t;H^{-1}(\Omega))$ such that $\partial_t v_t = \proj_{t,k-1}\partial_t v$ leads to
  \begin{align*}
    \sum_{K_t\in\cK_t} \min_{v_t\in S_k(\cK_t;H^{-1}(\Omega))} \|\partial_t(v-v_t)\|_{L^2(\patch(K_t);H^{-1}(\Omega))}^2 \lesssim \|(1-\proj_{t,k-1})\partial_t v\|_{L^2( I;H^{-1}(\Omega))},
  \end{align*}
  which finishes the proof.
\end{proof}

The next result is taken from~\cite[Theorem~17]{StevensonStorn23}. 
\begin{lemma}\label{lem:iop2}
  There exists a linear projection $\iop_2\colon L^2(Q)^d \to P_{k-1}(\cK_t)\otimes \RT_\ell(\cK_\bx)$ with 
  \begin{align} \label{eq:comm_iop2}
    \divx \circ \iop_2 = \proj\circ \divx \quad \text{on }L^2(I;\Hdivset\Omega).
  \end{align}
  For all $\btau \in L^2(Q)^d$, it holds that
  \begin{align}\label{eq:l2_iop2}
    \|\iop_2\btau\|_{L^2(Q)} \lesssim \|\btau\|_{L^2(Q)}
  \end{align}
  In addition, if $\btau\in L^2( I;H^1(\Omega)^d) \cap H^{1/2}( I; L^2(Q)^d)$, then
  \begin{align}\label{eq:app_iop2}
    \|\btau-\iop_2\btau\|_{L^2(Q)} \lesssim h_\bx \|\btau\|_{L^2( I; H^1(\Omega)^d)} + h_t^{1/2}\|\btau\|_{H^{1/2}( I; L^2(\Omega)^d)}.
  \end{align}
\end{lemma}
\begin{proof}
 The interpolation operator $\iop_2$ of~\cite[Theorem~17]{StevensonStorn23} satisfies the commuting diagram property as well as
   \begin{align*}
    \|\btau-\iop_2\btau\|_{L^2(Q)}^2 &\eqsim \sum_{K\in\cK}
    \Big( 
      \min_{\btau_\bx\in L^2(I;\RT_\ell(\cK_\bx))} \|\btau-\btau_\bx\|_{L^2(K_t;L^2(\patch(K_\bx)))}^2
      \\ 
      &\qquad + \min_{\btau_t \in P_{k-1}(I;L^2(K_x)^{d})} \|\btau-\btau_t\|_{L^2(K_t;L^2(K_\bx))}^2
    \Big)
  \end{align*}
  The boundedness estimate follows at once.
  The approximation property follows as in the proof of Lemma~\ref{lem:iop1}.
\end{proof}

In order to define a commuting quasi-interpolator on $U_h$, we follow~\cite{StevensonStorn23} and define for $\bv = (v,\btau)\in U$,  
\begin{align*}
  \iop\bv := (\iop_1 v,\iop_3\bv) \quad\text{with}\quad
  \iop_3\bv := \iop_2\big(\btau-\gradx(-\Delta_\bx)^{-1}(\partial_t(v-\iop_1 v))\big).
\end{align*}
We recall some properties from~\cite[Theorem~18]{StevensonStorn23}. The last statement is a combination with the aforegoing results from Lemma~\ref{lem:iop1} and~\ref{lem:iop2}. 
\begin{lemma}\label{lem:iopU}
  The operator $\iop\colon U\to U_h$ is a linear projection with 
  \begin{align}\label{eq:comm_iop}
    \divst\circ \iop = \proj\circ \divst.
  \end{align}
  For all $\bv = (v,\btau)\in U$, it holds that
  \begin{align}\label{eq:app_iop3}
    \|\btau-\iop_3\bv\|_{L^2(Q)} \lesssim \|\btau-\iop_2\btau\|_{L^2(Q)} + \|\partial_t(v-\iop_1 v)\|_{L^2( I;H^{-1}(\Omega))}.
  \end{align}
  If $v\in X\cap H^1( I; L^2(\Omega))$ and $\btau\in L^2( I; H^1(\Omega)^d)\cap H^{1/2}( I;L^2(\Omega)^d)$, then
  \begin{align}\label{eq:app2_iop3}
  \begin{split}
    \|\btau-\iop_3\bv\|_{L^2(Q)} &\lesssim h_\bx \|\grad_\bx\btau\|_{L^2(Q)} + h_t^{1/2}\|\btau\|_{H^{1/2}( I;L^2(\Omega)^d)} + h_\bx \|\partial_t v\|_{L^2(Q)}
    \\
    &\qquad + \|(1-\proj_{t,k-1})\partial_t v\|_{L^2( I;H^{-1}(\Omega))}.
  \end{split}
  \end{align}
  \qed
\end{lemma}

\section{Improved $L^2(Q)$ error estimates}\label{sec:l2_estimates}

In our first main result we show an improved a priori error estimate for the $L^2(Q)$ error in the primal variable of the FOSLS approximation.
\begin{theorem}\label{thm:L2est}
  Suppose that $\Omega$ is convex.
  Let $\bu\in U$ and $\bu_h\in U_h$ denote the solution of~\eqref{eq:minprob} and~\eqref{eq:fosls}, respectively.
 Then, 
  \begin{align*}
    \|u-u_h\|_{L^2(Q)} &\lesssim (h_\bx + h_t^{1/2})  
    (\|\gradx(u-u_h)\|_{L^2(Q)} + \|\bsigma-\bsigma_h\|_{L^2(Q)} + \|u(0)-u_h(0)\|_{L^2(\Omega)})
    \\
  &\qquad\quad+ (h_\bx^2+h_t)\|(1-\proj)f\|_{L^2(Q)}
  \\
  &\lesssim (h_\bx+h_t^{1/2}) \|\bu-\bu_h\|_U
  \end{align*}
  and, in particular, if $h_t\eqsim h_\bx^2$, then
  \begin{align*}
    \|u-u_h\|_{L^2(Q)} &\lesssim h_\bx (\|\gradx(u-u_h)\|_{L^2(Q)} + \|\bsigma-\bsigma_h\|_{L^2(Q)} + \|u(0)-u_h(0)\|_{L^2(\Omega)}) 
    \\
    &\qquad\quad+ h_\bx^2 \|(1-\proj)f\|_{L^2(Q)} \\
    &\lesssim h_\bx \|\bu-\bu_h\|_U.
  \end{align*}
\end{theorem}
\begin{proof}
  We split the proof into seven steps. 

  \noindent
  \textbf{Step 1 }(\emph{Dual problem})
  Let $v\in X$ denote the weak solution of the backward heat equation
  \begin{align*}
    -\partial_{t} v -\Delta_\bx v &= u-u_h, \\
    v(T) &= 0, \\
    v|_{I\times\partial\Omega} &= 0.
  \end{align*}
  Since $u-u_h\in L^2(Q)$, we have by parabolic regularity (Proposition~\ref{prop:parabolicreg} and Remark~\ref{rem:parabolicreg}) that $v\in Y$ with
  \begin{align*}
    \|v\|_Y \lesssim \|u-u_h\|_{L^2(Q)}.
  \end{align*}

  \noindent
  \textbf{Step 2 }(\emph{Primal-dual problem})
  Let $w\in X$ denote the weak solution of the heat equation
  \begin{align*}
    \partial_t w -\Delta_\bx w &= v-\Delta_\bx v, \\
    w(0) &= v(0), \\
    w|_{I\times\partial\Omega} &=0,
  \end{align*}
  and set $\bchi:= \gradx v - \gradx w$.
  Then, $\bw = (w,\bchi)\in U$ with $\divst\bw = v$.
  Recall that $v\in Y$, thus, $v(0)\in H_0^1(\Omega)$ and $\Delta_\bx v\in L^2(Q)$. 
  Parabolic regularity (Proposition~\ref{prop:parabolicreg}) yields $w\in Y$ with
  \begin{align*}
    \|w\|_Y \lesssim \|v-\Delta_\bx v \|_{L^2(Q)} + \|\gradx v(0)\|_{L^2(\Omega)} 
    \lesssim \|v\|_Y \lesssim \|u-u_h\|_{L^2(Q)},
  \end{align*}
  where the last estimate has been shown in Step~1. 
  Using $\bchi = \gradx v-\gradx w$ and Lemma~\ref{lem:embedding}, we further see that
  \begin{align*}
    \|\bchi\|_{H^{1/2}( I;L^2(\Omega)^d)} \lesssim \|v\|_{H^{1/2}(I;H^1(\Omega))} + \|w\|_{H^{1/2}(I;H^1(\Omega))} \lesssim  \|v\|_Y + \|w\|_Y \lesssim \|u-u_h\|_{L^2(Q)}.
  \end{align*}

  \noindent
  \textbf{Step 3 } (\emph{Duality arguments})
  With the solutions $v$ and $\bw = (w,\bchi)$ defined in Step~1--2, we are in position to develop duality arguments. 
  In the following, we use $\dual{\cdot}\cdot$ for the duality bracket of the pairing $L^2( I; H^{-1}(\Omega))\times L^2( I; H_0^1(\Omega))$.
  Integration by parts in time, the identity $\dual{\divx(\bsigma-\bsigma_h)}{v} + \ip{\bsigma-\bsigma_h}{\gradx v}_{L^2(Q)} = 0$,
  and Galerkin orthogonality~\eqref{eq:galerkinorth} yield that
  \begin{align*}
    \|u-u_h\|_{L^2(Q)}^2 &= \ip{u-u_h}{-\partial_t v -\Delta_\bx v}_{L^2(Q)}
    \\
    &=\dual{\partial_t(u-u_h)}{v} + \ip{(u-u_h)(0)}{v(0)}_{L^2(\Omega)}-\ip{(u-u_h)(T)}{v(T)}_{L^2(\Omega)}
    \\
    &\qquad +\ip{\gradx(u-u_h)}{\gradx v}_{L^2(Q)}
    \\
    &= \dual{\partial_t(u-u_h)}{v} + \ip{\gradx(u-u_h)}{\gradx v}_{L^2(Q)} + \ip{(u-u_h)(0)}{v(0)}_{L^2(\Omega)}
    \\
    &= \ip{\divst(\bu-\bu_h)}{v}_{L^2(Q)} + \ip{\gradx(u-u_h)+\bsigma-\bsigma_h}{\gradx v}_{L^2(Q)} 
    \\ &\qquad + \ip{(u-u_h)(0)}{v(0)}_{L^2(\Omega)}
    \\
    &= a(\bu-\bu_h,\bw) = a(\bu-\bu_h,\bw-\bw_h)
  \end{align*}
  for all $\bw_h\in U_h$.

  \noindent
  \textbf{Step 4 } (\emph{Commuting diagram})
  We choose $\bw_h := \iop\bw\in U_h$ in Step 3. Then, $\divst(\bw-\bw_h) = (1-\proj)\divst\bw = (1-\proj)v$ (see \eqref{eq:comm_iop})
  and $(1-\proj)\divst(\bu-\bu_h) = (1-\proj)\divst\bu = (1-\proj)f$ lead to
  \begin{align*}
    a(\bu-\bu_h,\bw-\bw_h) &= \ip{\divst(\bu-\bu_h)}{\divst(\bw-\bw_h)}_{L^2(Q)} 
    \\
    &\qquad + \ip{\gradx(u-u_h)+\bsigma-\bsigma_h}{\gradx (w-w_h)+\bchi-\bchi_h}_{L^2(Q)} 
    \\ &\qquad + \ip{(u-u_h)(0)}{(w-w_h)(0)}_{L^2(\Omega)}
    \\
    &= \ip{(1-\proj)f}{(1-\proj)v}_{L^2(Q)} 
    \\
    &\qquad + \ip{\gradx(u-u_h)+\bsigma-\bsigma_h}{\gradx (w-w_h)+\bchi-\bchi_h}_{L^2(Q)} 
    \\ &\qquad + \ip{(u-u_h)(0)}{(w-w_h)(0)}_{L^2(\Omega)}
    \\
    &\leq \|(1-\proj)f\|_{L^2(Q)}\|(1-\proj)v\|_{L^2(Q)} \\
    &\qquad + (\|\gradx(u-u_h)\|_{L^2(Q)} +\|\bsigma-\bsigma_h\|_{L^2(Q)})
    \\
    &\qquad\qquad \times (\|\gradx (w-w_h)\|_{L^2(Q)}+ \|\bchi-\bchi_h\|_{L^2(Q)})
    \\ &\qquad + \|(u-u_h)(0)\|_{L^2(\Omega)}\|(w-w_h)(0)\|_{L^2(\Omega)}
    \\
    &\lesssim \|(1-\proj)f\|_{L^2(Q)}\|(1-\proj)v\|_{L^2(Q)}
    \\
    &\qquad+ (\|\gradx(u-u_h)\|_{L^2(Q)} +\|\bsigma-\bsigma_h\|_{L^2(Q)}
    + \|(u-u_h)(0)\|_{L^2(\Omega)}) 
    \\
    &\qquad\qquad \times (\|w-w_h\|_X + \|\bchi-\bchi_h\|_{L^2(Q)}).
  \end{align*}
  In the last step we have used the continuous embedding~\eqref{eq:X2C0}.
  It remains to estimate the terms $\|(1-\proj)v\|_{L^2(Q)}$ and $\|w-w_h\|_X + \|\bchi-\bchi_h\|_{L^2(Q)}$ of the right-hand side, which is done in the next two steps. 

  \noindent
  \textbf{Step 5} (\emph{Estimate of $\|(1-\proj)v\|_{L^2(Q)}$})
  By Step~1, we know that $v\in Y$. Thus, the approximation properties of $\proj$ (Lemma~\ref{lem:L2proj}) give that
  \begin{align*}
    \|(1-\proj)v\|_{L^2(Q)} \lesssim h_\bx^2 \|D_\bx^2 v\|_{L^2(Q)} + h_t\|\partial_t v\|_{L^2(Q)}
    \lesssim (h_\bx^2+h_t)\|u-u_h\|_{L^2(Q)}.
  \end{align*}

  \noindent
  \textbf{Step 6} (\emph{Estimate of $\|w-w_h\|_X + \|\bchi-\bchi_h\|_{L^2(Q)}$})
  Using the definition of $\iop$ and Lemma~\ref{lem:iopU} (more precisely \eqref{eq:app_iop3}), we get that
  \begin{align*}
    \|w-w_h\|_X + \|\bchi-\bchi_h\|_{L^2(Q)} \lesssim \|w-\iop_1 w\|_X + \|\bchi-\iop_2\bchi\|_{L^2(Q)}.
  \end{align*}
  First, an application of Lemma~\ref{lem:iop1} (more precisely~\eqref{eq:app_grad_iop1}) yields that
  \begin{align*}
    \|\gradx(w-\iop_1w)\|_{L^2(Q)} \lesssim (h_\bx + h_t^{1/2}) \|w\|_Y
    \lesssim (h_\bx+h_t^{1/2})\|u-u_h\|_{L^2(Q)}.
  \end{align*}
  In order to estimate $\|\partial_t(w-\iop_1w)\|_{L^2( I;H^{-1}(\Omega))}$, we write (see Step~2) $\partial_t w = v-\divx\bchi$ and use Lemma~\ref{lem:iop1} (more precisely \eqref{eq:app_dt_iop1}) together with Lemma~\ref{lem:projP0time} to bound
  \begin{align*}
    \|\partial_t(w-\iop_1w)\|_{L^2( I;H^{-1}(\Omega))} 
    &\lesssim h_\bx \|\partial_t w\|_{L^2(Q)} + \|(1-\proj_{t,0})v\|_{L^2( I;H^{-1}(\Omega))} 
    \\
    &\qquad +\|(1-\proj_{t,0})\divx\bchi\|_{L^2( I;H^{-1}(\Omega))}
    \\
    &\lesssim h_\bx \|\partial_t w\|_{L^2(Q)} + h_t\|\partial_t v\|_{L^2( I;H^{-1}(\Omega))} 
    + \|(1-\proj_{t,0})\bchi\|_{L^2(Q)}
    \\
    &\lesssim h_\bx \|\partial_t w\|_{L^2(Q)} + h_t\|\partial_t v\|_{L^2(Q)} + h_t^{1/2}\|\bchi\|_{H^{1/2}( I;L^2(\Omega)^d)}
    \\
    &\lesssim (h_\bx + h_t + h_t^{1/2})(\|v\|_Y + \|w\|_Y) \lesssim (h_\bx+h_t^{1/2})\|u-u_h\|_{L^2(Q)}.
  \end{align*}
  For the remaining term $\|\bchi-\iop_2\bchi\|_{L^2(Q)}$, we employ Lemma~\ref{lem:iop2} (more precisely~\eqref{eq:app_iop2}) and $\bchi=\gradx v - \gradx w$ to obtain that
  \begin{align*}
    \|\bchi-\iop_2\bchi\|_{L^2(Q)} \lesssim h_\bx \|\gradx\bchi\|_{L^2(Q)}
    + h_t^{1/2} \|\bchi\|_{H^{1/2}( I;L^2(\Omega)^d)} \lesssim (h_\bx + h_t^{1/2})\|u-u_h\|_{L^2(Q)}.
  \end{align*}

  \noindent
  \textbf{Step 7} (\emph{Conclusion})
  In the final step, we combine the aforegoing results and conclude that
  \begin{align*}
    \|u-u_h\|_{L^2(Q)}^2 &\lesssim \|(1-\proj)f\|_{L^2(Q)}\|(1-\proj)v\|_{L^2(Q)}
    \\
    &\qquad (\|\gradx(u-u_h)\|_{L^2(Q)} +\|\bsigma-\bsigma_h\|_{L^2(Q)}
    + \|(u-u_h)(0)\|_{L^2(\Omega)}) 
    \\
    &\qquad\qquad \times (\|w-w_h\|_X + \|\bchi-\bchi_h\|_{L^2(Q)})
    \\
    &\lesssim \|u-u_h\|_{L^2(Q)}\Big((h_\bx^2+h_t)\|(1-\proj)f\|_{L^2(Q)} 
      \\
      &\qquad+(h_\bx+h_t^{1/2})(\|\gradx(u-u_h)\|_{L^2(Q)} +\|\bsigma-\bsigma_h\|_{L^2(Q)}
    + \|(u-u_h)(0)\|_{L^2(\Omega)})\Big).
  \end{align*}

  By \eqref{eq:X2C0} and \eqref{eq:X2U}, we have that
  \begin{align*}
    \|\gradx(u-u_h)\|_{L^2(Q)} + \|\bsigma-\bsigma_h\|_{L^2(Q)} + \|u(0)-u_h(0)\|_{L^2(\Omega)} \lesssim \|\bu-\bu_h\|_U.
  \end{align*}
  Furthermore, using the properties of the $L^2(Q)$ projection $\proj$, we see that
  \begin{align*}
    \|(1-\proj)f\|_{L^2(Q)} &= \|(1-\proj)\divst\bu\|_{L^2(Q)} = \|(1-\proj)\divst(\bu-\bu_h)\|_{L^2(Q)}
    \\
    &\leq \|\divst(\bu-\bu_h)\|_{L^2(Q)} \leq \|\bu-\bu_h\|_U.
  \end{align*}
  This finishes the proof.
\end{proof}

Our second main result states a higher-order conservation property.
\begin{theorem}\label{thm:conservation}
  Suppose that $\Omega$ is convex.
  Let $\bu\in U$ and $\bu_h\in U_h$ denote the solution of~\eqref{eq:minprob} and~\eqref{eq:fosls}, respectively.
 Then, 
  \begin{align*}
    \|\proj f - \divst\bu_h\|_{L^2(Q)} &\lesssim (h_\bx + h_t^{1/2})  
    (\|\gradx(u-u_h)\|_{L^2(Q)} + \|\bsigma-\bsigma_h\|_{L^2(Q)} + \|u(0)-u_h(0)\|_{L^2(\Omega)})
  \end{align*}
  and, in particular, if $h_t\eqsim h_\bx^2$, then
  \begin{align*}
    \|\proj f - \divst\bu_h\|_{L^2(Q)} &\lesssim h_\bx (\|\gradx(u-u_h)\|_{L^2(Q)} + \|\bsigma-\bsigma_h\|_{L^2(Q)} + \|u(0)-u_h(0)\|_{L^2(\Omega)}).
  \end{align*}
\end{theorem}
\begin{proof}
  We start by defining $\bv = (v,\btau)\in U$ with $\btau := -\gradx v$ and $v\in X$ the solution of
  \begin{align*}
    \partial_t v -\Delta_\bx v &= \proj f -\divst\bu_h, \\
    v(0) &= 0, \\
    v|_{I\times\partial\Omega} &= 0. 
  \end{align*}
  By parabolic regularity (Proposition~\ref{prop:parabolicreg}), we have $v\in Y$ and 
  \begin{align*}
    \|v\|_Y + \|\btau\|_{H^{1/2}( I;L^2(\Omega)^d)} \lesssim \|v\|_Y \lesssim \|\proj f -\divst\bu_h\|_{L^2(Q)}.
  \end{align*}
  Note that $\ip{\proj f -\divst\bu_h}{p}_{L^2(Q)} = \ip{f-\divst\bu_h}p_{L^2(Q)}$ for $p\in P_{k-1}(\cK_t)\otimes P_\ell(\cK_\bx)$.
  Together with $f = \divst\bu$ and Galerkin orthogonality~\eqref{eq:galerkinorth}, we infer that
  \begin{align*}
    \|\proj f-\divst\bu_h\|_{L^2(Q)}^2 &= \ip{\proj f-\divst\bu_h}{\partial_t-\Delta_\bx v}_{L^2(Q)}
    \\
    &= \ip{\divst(\bu-\bu_h)}{\divst\bv}_{L^2(Q)} + \ip{\gradx(u-u_h)+\bsigma-\bsigma_h}{\gradx v +\btau}_{L^2(Q)}
    \\
    &\qquad + \ip{(u-u_h)(0)}{v(0)}_{L^2(\Omega)}
    \\
    &= a(\bu-\bu_h,\bv) = a(\bu-\bu_h,\bv-\bv_h)
  \end{align*}
  for all $\bv_h\in U_h$. 
  Choosing $\bv_h := \iop\bv$, the commuting-diagram property~\eqref{eq:comm_iop} implies $\divst(\bv-\bv_h) = (1-\proj)\divst\bv = 0$ so that
  \begin{align*}
    a(\bu-\bu_h,\bv-\bv_h) &= \ip{\gradx(u-u_h)+\bsigma-\bsigma_h}{\gradx (v-v_h) +\btau-\btau_h}_{L^2(Q)}
    \\
    &\qquad + \ip{(u-u_h)(0)}{(v-v_h)(0)}_{L^2(\Omega)}
    \\
    &\lesssim (\|\gradx(u-u_h)\|_{L^2(Q)} + \|\bsigma-\bsigma_h\|_{L^2(Q)} + \|(u-u_h)(0)\|_{L^2(\Omega)})
    \\
    &\qquad \times (\|v-v_h\|_X + \|\btau-\btau_h\|_{L^2(Q)}).
  \end{align*}
  It only remains to estimate $\|v-v_h\|_X + \|\btau-\btau_h\|_{L^2(Q)}$. This can be done as in Step~6 of the proof of Theorem~\ref{thm:L2est}. We first obtain that
  \begin{align*}
    \|\gradx(v-v_h)\|_{L^2(Q)} + \|\btau-\btau_h\|_{L^2(Q)} \lesssim (h_\bx + h_t^{1/2})\|\proj f - \divst\bu_h\|_{L^2(Q)}.
  \end{align*}
  The treatment of the term $\|\partial_t(v-v_h)\|_{L^2(I;H^{-1}(\Omega))}$ requires an intermediate step.
  With $\partial_t v = \proj f-\divst\bu_h -\divx\btau$ and $(1-\proj_{t,k-1})(\proj f-\divst\bu_h)=0$, we get by~\eqref{eq:app_dt_iop1} that
  \begin{align*}
  \|\partial_t(v-v_h)\|_{L^2(I;H^{-1}(\Omega))} &\lesssim h_\bx \|\partial_t v\|_{L^2(Q)} + \|(1-\proj_{t,k-1})\divx\btau\|_{L^2(I;H^{-1}(\Omega))}.
\end{align*}
Now, $\|\partial_t(v-v_h)\|_{L^2(I;H^{-1}(\Omega))}\lesssim (h_\bx+h_t^{1/2})\|\proj f-\divst\bu_h\|_{L^2(Q)}$ follows with the remaining arguments of Step~6, which finishes the proof.
\end{proof}

\section{$L^\infty(I;L^2(\Omega))$ error estimate}\label{sec:observations}
Numerical experiments indicate that $\|\bsigma-\bsigma_h\|_{L^2(Q)}$ and $\|u(T)-u_h(T)\|_{L^2(\Omega)}$ with $\bu = (u,\bsigma)$ the solution to~\eqref{eq:EL} and $\bu_h=(u_h,\bsigma_h)\in U_h$ the solution to~\eqref{eq:EL:disc}, converge at a higher rate than the overall error $\|\bu-\bu_h\|_U$. 
While we have not found a proof for this observation, we have the following partial explanation. 

\begin{proposition}\label{prop:EndTime}
  Let $\bu\in U$ and $\bu_h\in U_h$ denote the solution of~\eqref{eq:minprob} and~\eqref{eq:fosls}, respectively.
  Then,
  \begin{align*}
    \max_{\tau\in[0,T]}\|u(\tau)-u_h(\tau)\|_{L^2(\Omega)}^2 &\lesssim \|f-\divst\bu_h\|_{L^2(Q)}\|u-u_h\|_{L^2(Q)} + \|u-u_h\|_{L^2(Q)}^2 \\
    &\qquad + 
    \|\gradx(u-u_h)\|_{L^2(Q)}\|\bsigma-\bsigma_h\|_{L^2(Q)}\\
    &\lesssim \|\bu-\bu_h\|_U(\|u-u_h\|_{L^2(Q)} + \|\bsigma-\bsigma_h\|_{L^2(Q)}).
  \end{align*}
\end{proposition}
\begin{proof}
  Fix $\tau\geq T/2$ and define $Q_{\tau} := (0,\tau)\times \Omega$ and $\eta_\tau(t,\bx):= t/\tau$ for $(t,\bx)\in\overline Q_\tau = [0,\tau]\times\overline\Omega$.  
    We note that $\eta_\tau(0,\cdot)=0$, $\eta_\tau(\tau,\cdot)=1$, and $\partial_t \eta_\tau =  1/\tau \le 2/T$.
  Integration by parts in time and $\dual{\divx(\bsigma-\bsigma_h)}{\eta_\tau(u-u_h)}_{Q_{\tau}} = -\ip{\bsigma-\bsigma_h}{\eta_\tau\gradx(u-u_h)}_{L^2(Q_{\tau})}$ 
  (here, $\dual\cdot\cdot_{Q_{\tau}}$ denotes the duality pairing of $L^2((0,\tau);H^{-1}(\Omega))\times L^2( (0,\tau);H_0^1(\Omega))$) show that
  \begin{align*}
  \|u(\tau)-u_h(\tau)\|_{L^2(\Omega)}^2 &= \ip{\eta_\tau(\tau)(u-u_h)(\tau)}{(u-u_h)(\tau)}_{L^2(\Omega)}
    \\
  &= \dual{\partial_t(\eta_\tau(u-u_h))}{u-u_h}_{Q_\tau} + \dual{\eta_\tau(u-u_h)}{\partial_t(u-u_h)}_{Q_\tau}
    \\
    &= \ip{\partial_t(\eta_\tau)(u-u_h)}{u-u_h}_{L^2(Q_{\tau})} + \dual{\eta_\tau\partial_t(u-u_h)}{u-u_h}_{Q_\tau} \\ &\qquad+ \dual{\eta_\tau(u-u_h)}{\partial_t(u-u_h)}_{Q_\tau}
  \\
&= \ip{\partial_t(\eta_\tau)(u-u_h)}{u-u_h}_{L^2(Q_{\tau})} + 2\dual{\partial_t(u-u_h)}{\eta_\tau(u-u_h)}_{Q_\tau}
    \\
    &= \ip{\partial_t(\eta_\tau)(u-u_h)}{u-u_h}_{L^2(Q_{\tau})} + 2\ip{\divst(\bu-\bu_h)}{\eta_\tau(u-u_h)}_{L^2(Q_{\tau})} \\
    &\qquad + 2\ip{\bsigma-\bsigma_h}{\eta_\tau\gradx(u-u_h)}_{L^2(Q_{\tau})}.
\end{align*}
  The Cauchy--Schwarz inequality, $\divst\bu = f$, $\|\eta_\tau\|_{L^\infty(Q_{\tau})}=1$, and $\|\partial_t\eta_\tau\|_{L^\infty(Q_{\tau})}\le 2/T$ then show that
  \begin{align*}
    \|u(\tau)-u_h(\tau)\|_{L^2(\Omega)}^2 &\lesssim \|u-u_h\|_{Q_{\tau}}^2 + \|\divst(\bu-\bu_h)\|_{L^2(Q_{\tau})}\|u-u_h\|_{L^2(Q_{\tau})}
    \\
    &\qquad + \|\bsigma-\bsigma_h\|_{L^2(Q_{\tau})}\|\gradx(u-u_h)\|_{L^2(Q_{\tau})}.
  \end{align*}
  Using $f = \divst\bu$ and $\|\cdot\|_{L^2(Q_{\tau})} \leq \|\cdot\|_{L^2(Q)}$ proves the assertion for $\tau\geq T/2$. 
In the case $\tau<T/2$ one argues similarly, defining $Q_{\tau} := (\tau,T)\times \Omega$ and $\eta_\tau(t,\bx):= (T-t)/(T-\tau)$ for $(t,\bx)\in\overline Q_\tau = [\tau,T]\times\overline\Omega$.  
\end{proof}

\begin{remark}\label{rem:sigma}
  Suppose that $\Omega$ is convex.
  By Theorem~\ref{thm:L2est}, we know that $\|u-u_h\|_{L^2(Q)}$ converges at a higher rate than $\|\bu-\bu_h\|_U$. 
  If we assume that $\|\bsigma-\bsigma_h\|_{L^2(Q)}$ converges at a higher rate, then Proposition~\ref{prop:EndTime} states that $\|u(\tau)-u_h(\tau)\|_{L^2(\Omega)}$ converges at a higher rate than $\|\bu-\bu_h\|_U$.
\end{remark}

\begin{remark}
  For certain elliptic problems, it was shown in~\cite[Lemma~4.1 and Theorem~4.4]{Ku11} that the $L^2$ error of the vector-valued variable converges at the same rate as the $L^2$ quasi-interpolation error of the vector-valued variable plus the $L^2$ error of the scalar-valued variable.
  However, the corresponding proof is not directly applicable to FOSLS for parabolic PDEs.
\end{remark}

\section{Notes on quasi-optimality in weaker norms}\label{sec:quasi_optimality}

In this section, we discuss observations on quasi-optimality of approximations in weaker norms. 
The presented results do not rely on duality arguments and are thus independent of the results obtained in the sections above.
In particular, we do not require a parabolic regularity shift.

Let $\projHmOne$ denote a projection operator onto the space $P_{k-1}(\cK_t)\otimes P_\ell(\cK_\bx)$. 
We consider the minimization problem
\begin{align}\label{eq:fosls:modified}
\bu_h = (u_h,\bsigma_h) = \argmin_{\bv_h=(v_h,\btau_h)\in U_h} \|\divst\bv_h-\projHmOne f\|_{L^2(Q)}^2 + 
\|\gradx v_h+\btau_h\|_{L^2(Q)}^2 + \|v_h(0)-u_0\|_{L^2(\Omega)}^2.
\end{align}
Note that problem~\eqref{eq:fosls:modified} is simply~\eqref{eq:fosls} with $f$ replaced by $\projHmOne f$.
The following result follows a similar line of argumentation as in~\cite[Theorem~6]{FuehrerMixedFEMFOSLS}.
For projections bounded in $L^2( I;H^{-1}(\Omega))$, it has already been mentioned briefly in~\cite[Remark~19]{StevensonStorn23}.

\begin{theorem}\label{thm:quasiopt}
  Suppose that $f\in L^2(Q)$ and $u_0\in L^2(\Omega)$.
  Let $u\in X$ denote the solution of~\eqref{eq:heat}, $\bsigma:=-\nabla_x u$, and $\bu_h\in U_h$ the solution of~\eqref{eq:fosls:modified}. 
  Then,
  \begin{align*}
    \|u-u_h\|_X + \|\bsigma-\bsigma_h\|_{L^2(Q)} \lesssim \min_{(v_h,\btau_h)\in U_h}
    \|u-v_h\|_X + \|\bsigma-\btau_h\|_{L^2(Q)} + \|(1-\projHmOne)f\|_{L^2( I;H^{-1}(\Omega))}.
  \end{align*}
  If $\projHmOne$ is bounded in $L^2( I;H^{-1}(\Omega))$, then it holds even for $f\in L^2(I;H^{-1}(\Omega))$ that
  \begin{align*}
    \|u-u_h\|_X + \|\bsigma-\bsigma_h\|_{L^2(Q)} \lesssim \min_{(v_h,\btau_h)\in U_h}
    \|u-v_h\|_X + \|\bsigma-\btau_h\|_{L^2(Q)}.
  \end{align*}
\end{theorem}
\begin{proof}
  Define $u^h\in X$ as the solution to heat equation
  \begin{align*}
    \partial_t u^h - \Delta_\bx u^h &= \projHmOne f, \\
    u^h(0) &= u_0, \\
    u^h|_{{I\times\partial\Omega}} &= 0,
  \end{align*}
  and set $\bsigma^h := -\gradx u^h$. By stability of the inverse of the heat operator~\eqref{eq:X2RHS}, we obtain that
  \begin{align}\label{eq:proofquasiopt:a}
    \|u-u^h\|_X + \|\bsigma-\bsigma^h\|_{L^2(Q)} \lesssim \|(1-\projHmOne)f\|_{L^2( I;H^{-1}(\Omega))}.
  \end{align}
  Note that the solution $\bu_h = (u_h,\bsigma_h)\in U_h$ of~\eqref{eq:fosls:modified} is an approximation to $\bu^h = (u^h,\bsigma^h)\in U$ and by quasi-optimality of the FOSLS, we infer that
  \begin{align*}
    \|u^h-u_h\|_X + \|\bsigma^h-\bsigma_h\|_{L^2(Q)} \lesssim \min_{\bv_h\in U_h} \|\bu^h-\bv_h\|_U.
  \end{align*}
  The last observation together with $\bu^h-\iop \bu^h = (1-\iop)(\bu^h-\bv_h)$ for any $\bv_h = (v_h,\btau_h)\in U_h$, $\divst(1-\iop)(\bu^h-\bv_h) = \divst(\bu^h-\iop\bu^h) = \projHmOne f - \proj\projHmOne f = 0$ (see~\eqref{eq:comm_iop}), boundedness~\eqref{eq:grad_iop1}, \eqref{eq:dt_iop1}, \eqref{eq:l2_iop2} \& \eqref{eq:app_iop3} of the involved interpolation operators, and \eqref{eq:proofquasiopt:a} proves that
  \begin{align*}
    \|u^h-u_h\|_X + \|\bsigma^h-\bsigma_h\|_{L^2(Q)} &\lesssim \|(1-\iop)(\bu^h-\bv_h)\|_U 
    \\
    &\lesssim \|\gradx(1-\iop_1)(u^h-v_h)\|_{L^2(Q)} + \| (1-\iop_3) (\bsigma^h - \btau)\|_{L^2(Q)} 
    \\
    &\lesssim \|\gradx(u^h-v_h)\|_{L^2(Q)} + \|\bsigma^h-\btau_h\|_{L^2(Q)} + \|\partial_t(u^h-v_h)\|_{L^2(I;H^{-1}(\Omega))}
    \\
    &\lesssim \|u-v_h\|_X + \|\bsigma-\btau_h\|_{L^2(Q)} + \|u-u^h\|_X + \|\bsigma-\bsigma^h\|_{L^2(Q)}
    \\
    &\lesssim \|u-v_h\|_X + \|\bsigma-\btau_h\|_{L^2(Q)} + \|(1-\projHmOne)f\|_{L^2( I; H^{-1}(\Omega))}.
  \end{align*}
  This shows the first statement. 

  The statement is still true with the same argumentation if $\projHmOne$ is bounded in $L^2( I; H^{-1}(\Omega))$ and $f\in L^2( I; H^{-1}(\Omega))$. 
  It remains to verify that
  \begin{align*}
    \|(1-\projHmOne)f\|_{L^2( I;H^{-1}(\Omega))} \lesssim \min_{(v_h,\btau_h)\in U_h}
    \|u-v_h\|_X + \|\bsigma-\btau_h\|_{L^2(Q)}
  \end{align*}
  Let $\bv_h =(v_h,\btau_h) \in U_h$ be given. Using $f = \divst\bu$, boundedness of $\projHmOne$, and $\divx\colon L^2(Q)^d \to L^2( I; H^{-1}(\Omega))$, we conclude that
  \begin{align*}
    \|(1-\projHmOne)f\|_{L^2( I;H^{-1}(\Omega))} &= \|(1-\projHmOne)(\divst\bu-\divst\bv_h)\|_{L^2( I;H^{-1}(\Omega))}
    \\
    &\lesssim \|\divst(\bu-\bv_h)\|_{L^2( I; H^{-1}(\Omega))} 
    \\
    &\leq \|\partial_t(u-v_h)\|_{L^2( I;H^{-1}(\Omega))} 
    + \|\divx(\bsigma-\btau_h)\|_{L^2( I;H^{-1}(\Omega))}
    \\
    &\lesssim \|u-v_h\|_X + \|\bsigma-\btau_h\|_{L^2(Q)}.
  \end{align*}
  Since $(v_h,\btau_h)\in U_h$ was arbitrary, this finishes the proof.
\end{proof}

\begin{remark}
  Again, the result can be readily extended to general second-order parabolic PDEs; see~\cite{GantnerStevenson21} for the least-squares formulations.
\end{remark}

\begin{remark}
  If we choose $\projHmOne = \proj$, then minimization problem~\eqref{eq:fosls:modified} is equivalent to~\eqref{eq:fosls}, since, 
  \begin{align*}
    \bu_h = (u_h,\bsigma_h) &= \argmin_{\bv=(v,\btau)\in U_h} \|\divst\bv- f\|_{L^2(Q)}^2 + 
    \|\gradx v+\btau\|_{L^2(Q)}^2 + \|v(0)-u_0\|_{L^2(\Omega)}^2\\
    &= \argmin_{\bv=(v,\btau)\in U_h} \|\divst\bv-\proj f\|_{L^2(Q)}^2 + 
    \|\gradx v+\btau\|_{L^2(Q)}^2 + \|v(0)-u_0\|_{L^2(\Omega)}^2\\
    &\qquad + \|(1-\proj)f\|_{L^2(Q)}^2 \\
    &= \argmin_{\bv=(v,\btau)\in U_h} \|\divst\bv-\projHmOne f\|_{L^2(Q)}^2 + 
    \|\gradx v+\btau\|_{L^2(Q)}^2 + \|v(0)-u_0\|_{L^2(\Omega)}^2.
  \end{align*}
  Therefore, Theorem~\ref{thm:quasiopt} states that the FOSLS~\eqref{eq:fosls} is almost (up to the data oscillation term $\|(1-\proj)f\|_{L^2(Q)}$) quasi-optimal in a norm which is weaker than $\|\cdot\|_U$.
\end{remark}

\begin{remark}
  One possibility to construct an $L^2( I;H^{-1}(\Omega))$-stable operator $\projHmOne$ is to combine $\proj_{t,k-1}$, $k\in\N$, with an $H^{-1}(\Omega)$-stable projection from $H^{-1}(\Omega)$ to $P_{\ell}(\cK_\bx)$, $\ell\in\N_0$, see, e.g.,~\cite{FuehrerMixedFEMFOSLS} for the lowest-order case $\ell=0$ and~\cite{MR4570332} for $\ell\geq 1$.
\end{remark}

\section{Numerical experiments}\label{sec:numerics}

In this section, we numerically verify our theoretical findings. 
We apply the space-time FOSLS 
\begin{align*}
  \bu_h = (u_h,\bsigma_h) = \argmin_{\bv=(v,\btau)\in U_h} \|\divst\bv-f\|_{L^2(Q)}^2 + 
  \|\gradx v+\btau\|_{L^2(Q)}^2 + \|v(0)-u_0\|_{L^2(\Omega)}^2
\end{align*}
(presented in detail in Section~\ref{sec:fosls}) to prescribed solutions of the heat equation on the unit interval $\Omega = (0,1)$ and the unit square $\Omega = (0,1)^2$, respectively, with end time $T=1$. 
As ansatz space, we consider 
\begin{align*}
  U_h := S_{1}(\cK_t)\otimes S_{1,0}(\cK_\bx) \times P_{0}(\cK_t)\otimes \RT_1(\cK_\bx),
\end{align*}
i.e., we choose $k=1$ and $\ell=1$ in \eqref{eq:U_h}, 
on uniformly refined tensor-meshes $\cK$. 
We choose initial meshes $\cK_\bx:=\{(0,1)\}$ and $\cK_\bx:=\{{\rm conv}\{(0,0),(1,0),(0,1)\},$ ${\rm conv}\{(1,0),(1,1),(0,1)\}\}$, respectively, as well as $\cK_t:=\{(0,1)\}$.
Here, ${\rm conv}$ denotes the open convex hulls of a set of points. 
In spatial direction, we employ bisection and red refinement (i.e., triangles are split into four new triangles by connecting the midpoints of their edges), respectively. 
In temporal direction, we employ either one bisection or two bisections to generate either a sequence of uniform \emph{equally} scaled meshes, i.e., $h_t \eqsim h_\bx$, or a sequence of uniform \emph{parabolically} scaled meshes, i.e., $h_t\eqsim h_\bx^2$. 
We write $s:=1$ for equally scaled meshes and $s:=2$ for parabolically scaled meshes. 
Then, we have that $h_t\eqsim h_\bx^s$. 

For smooth solutions $\bu = (u,-\gradx u)$, we expect the convergence rate 
\begin{align}\label{eq:a_priori}
  \| \bu - \bu_h\|_U = \mathcal{O}(h_\bx) = \mathcal{O}({\rm dofs}^{-1/(d+s)}),
\end{align}
where ${\rm dofs}$ denotes the number of degrees of freedom, i.e., the dimension of $U_h$; cf.~\cite{GantnerStevenson24}.
We compute the (squared) least-squares error 
\begin{align*}
  LS(f,u_0;\bu_h)^2 := \| f- \divst\bu_h\|_{L^2(Q)}^2 + \|\gradx u_h+\bsigma_h\|_{L^2(Q)}^2 + \|u_0 - u_h(0)\|_{L^2(\Omega)}^2
  =a(\bu - \bu_h,\bu - \bu_h),
\end{align*}
which is equivalent to the (squared) error $\| \bu - \bu_h\|_U^2$ by Proposition~\ref{prop:fosls}, 
the $L^2$ errors 
\begin{align*}
  \| u - u_h \|_{L^2(Q)}, \quad  \| \bsigma - \bsigma_h \|_{L^2(Q)},   \quad \| u(T) - u_h(T) \|_{L^2(\Omega)}, 
\end{align*}
and 
\begin{align*}
  \| \Pi f -\divst\bu_h\|_{L^2(Q)}. 
\end{align*}
According to Theorem~\ref{thm:L2est} and~\ref{thm:conservation}, we have for smooth solutions that
\begin{align}\label{eq:a_priori_l2}
  \| u - u_h \|_{L^2(Q)} = \mathcal{O}(h_\bx^2) = \mathcal{O}({\rm dofs}^{-2/(d+s)}) = \| \Pi f -\divst\bu_h\|_{L^2(Q)}.
\end{align}
In particular, from standard approximation theory, we thus have for smooth $f$ that 
\begin{align}\label{eq:a_priori_f}
  \| f -\divst\bu_h\|_{L^2(Q)} = \mathcal{O}(h_\bx^s) = \mathcal{O}({\rm dofs}^{-s/(d+s)}).
\end{align}
Finally, since $\bsigma_h\in P_{0}(\cK_t)\otimes \RT_1(\cK_\bx)$, we may ``at best'' expect that
\begin{align}\label{eq:a_priori_sigma}
  \| \bsigma - \bsigma_h \|_{L^2(Q)} = \mathcal{O}(h_\bx^s) = \mathcal{O}({\rm dofs}^{-s/(d+s)}).
\end{align}
In this case, Proposition~\ref{prop:EndTime} predicts that 
\begin{align}\label{eq:a_priori_uT}
  \| u(0) - u_h(0) \|_{L^2(\Omega)} = \mathcal{O}(h_\bx^{(1+s)/2}) = \mathcal{O}({\rm dofs}^{-(1+s)/2(d+s)}) = \| u(T) - u_h(T) \|_{L^2(\Omega)}.
\end{align}

\definecolor{darkgreen}{rgb}{0.0, 0.5, 0.0}
\definecolor{darkyellow}{rgb}{0.8, 0.7, 0.0}
\newcommand{\plotMyFigure}[5]{
\begin{figure}[ht]
\centering
\begin{tikzpicture}
\pgfplotstableread[col sep=comma]{data/#1D/#2_par#3_pxs#4.csv}{\mydata}
\begin{loglogaxis}[
width = 0.645\textwidth,
xlabel={degrees of freedom},
xmajorgrids=true, ymajorgrids=true, 
legend style={
legend pos=outer north east,
align=left,
fill=none, draw=none,
font=\tiny
}
]
\addplot[red, mark=*, very thick] table[x=dofs,y=estimators]{\mydata};
\addplot[blue, mark=square, very thick] table[x=dofs,y=estimators_f]{\mydata};
\addplot[darkgreen, mark=x, very thick] table[x=dofs,y=estimators_g]{\mydata};
\addplot[magenta, mark=triangle, very thick] table[x=dofs,y=estimators_u0]{\mydata};
\addplot[cyan, mark=diamond, very thick] table[x=dofs,y=estimators_u]{\mydata};
\addplot[blue, mark=square, very thick, dashed, mark options={solid}] table[x=dofs,y=estimators_Pf]{\mydata};
\addplot[darkgreen, mark=x, very thick, dashed, mark options={solid}] table[x=dofs,y=estimators_sigma]{\mydata};
\addplot[darkyellow, mark=triangle, mark options={rotate=180}, very thick] table[x=dofs,y=estimators_u1]{\mydata};

\pgfplotstablegetelem{0}{dofs}\of\mydata
\let\firstdof\pgfplotsretval

\pgfplotstableset{create on use/rowcount/.style={create col/expr={\pgfplotstablerows}}}
\pgfplotstablecreatecol[create col/expr={\pgfplotstablerows}]{dummy}{\mydata}
\pgfplotstablegetelem{0}{dummy}\of\mydata
\let\rowcount\pgfplotsretval
\pgfmathtruncatemacro{\lastrow}{\rowcount - 1}
\pgfplotstablegetelem{\lastrow}{dofs}\of\mydata
\let\lastdof\pgfplotsretval

#5 

\legend{
$LS(u_h{,}\bsigma_h)$,
$\| f -\divst\bu_h\|_{L^2(Q)}$,
$\|\gradx u_h+\bsigma_h\|_{L^2(Q)}$,
$\|(u-u_h)(0)\|_{L^2(\Omega)}$,
$\| u - u_h\|_{L^2(Q)}$,
$\| \Pi f -\divst\bu_h\|_{L^2(Q)}^2$,
$\| \bsigma - \bsigma_h\|_{L^2(Q)}$,
$\|(u-u_h)(T)\|_{L^2(\Omega)}$,
}
\end{loglogaxis}
\end{tikzpicture}
\caption{Convergence plot for \ifnum\pdfstrcmp{#2}{f0_u0hat}=0 non-\fi \ifnum\pdfstrcmp{#2}{square_f0_u0hat}=0 non-\fi smooth solution \ifnum\pdfstrcmp{#1}{1}=0 on unit interval \else on unit square \fi  of Section~\ref{sec:#1d_#2} for \ifnum\pdfstrcmp{#3}{0}=0 equally scaled meshes with $h_t\eqsim h_\bx$\else parabolic meshes with $h_t\eqsim h_\bx^2$\fi. 
}
\label{fig:#1d_#2_par#3_pxs#4}
\end{figure}
}

\subsection{Smooth solution on unit interval}\label{sec:1d_sincos}

We consider the solution of the heat equation~\eqref{eq:heat} 
\begin{align*}
  u(t,x) := \cos(\pi t) \sin(\pi x) \quad \forall (t,x) \in \overline Q = \overline I\times \overline \Omega = [0,1]^2,
\end{align*}
and choose the data $f$ and $u_0$ accordingly. 
In Figure~\ref{fig:1d_sincos_par0_pxs2}--\ref{fig:1d_sincos_par1_pxs2}, we provide the double-logarithmic convergence plots 
on equally scaled and parabolically scaled meshes, respectively. 
In both cases, we observe the expected convergence rate $\mathcal{O}(h_\bx)=\mathcal{O}({\rm dofs}^{-1/(1+s)})$ of~\eqref{eq:a_priori} for the overall least-squares error $LS(f,u_0;\bu_h)$ as well as the component $\|\gradx u_h+\bsigma_h\|_{L^2(Q)}$.
The errors $\| u - u_h \|_{L^2(Q)}$ and $\| \Pi f -\divst\bu_h\|_{L^2(Q)}$ converge both at the double rate $\mathcal{O}(h_\bx^2)=\mathcal{O}({\rm dofs}^{-2/(1+s)})$, which is $1/2$ order of $h_\bx$ better than the expected rate~\eqref{eq:a_priori_l2} for $s=1$. 
We obtain the expected rate $\| f -\divst\bu_h\|_{L^2(Q)} = \mathcal{O}(h_\bx^s) = \mathcal{O}({\rm dofs}^{-s/(d+s)})$ of~\eqref{eq:a_priori_f}. 
For $\| \bsigma - \bsigma_h \|_{L^2(Q)}$, we observe the ``best'' rate $\mathcal{O}(h_\bx^s) = \mathcal{O}({\rm dofs}^{-s/(d+s)})$ of~\eqref{eq:a_priori_sigma}.
In both cases, $\|u(0)-u_h(0)\|_{L^2(\Omega)}$ and $\|u(T)-u_h(T)\|_{L^2(\Omega)}$ converge at rate $\mathcal{O}(h_\bx^2)=\mathcal{O}({\rm dofs}^{-2/(1+s)})$, which is  $1$ or $1/2$ order of $h_\bx$  better than the expected rate~\eqref{eq:a_priori_uT} for $s=1$ and $s=2$, respectively. 


\plotMyFigure{1}{sincos}{0}{2}{
	\addplot[black,dashed,domain=\firstdof:\lastdof] {20*x^(-1/2) };
	\node at (axis cs:1e4,1.0e0) [anchor=north west] {$\mathcal{O}({\rm dofs}^{-1/2})$};
	\addplot[black,dashed,domain=\firstdof:\lastdof] {0.4*x^(-1) };
	\node at (axis cs:2*1e3,1.0e-5) [anchor=north west] {$\mathcal{O}({\rm dofs}^{-1})$};
}


\plotMyFigure{1}{sincos}{1}{2}{
	\addplot[black,dashed,domain=\firstdof:\lastdof] {3.5*x^(-1/3) };
	\node at (axis cs:1e5,3*1.0e-1) [anchor=north west] {$\mathcal{O}({\rm dofs}^{-1/3})$};
	\addplot[black,dashed,domain=\firstdof:\lastdof] {0.4*x^(-2/3) };
	\node at (axis cs:5*1e3,1.0e-4) [anchor=north west] {$\mathcal{O}({\rm dofs}^{-2/3})$};
}

\subsection{Non-smooth solution on unit interval}\label{sec:1d_f0_u0hat}

We choose 
\begin{align*}
  f(t,x) := 0 \quad\forall (t,x)\in\overline Q = \overline I\times \overline \Omega = [0,1]^2
  \quad \text{and} \quad
  u_0(x) := 1-2|x-\tfrac12| \quad\forall x\in\overline \Omega = [0,1].
\end{align*}
The corresponding solution of the heat equation~\eqref{eq:heat} can be expressed as the Fourier series
\begin{align*}
  u(t,x) = \sum_{n=1}^\infty \Big[2 \int_0^1 u_0(y) \sin(n\pi y) \,{\rm d}y\Big]\sin(n\pi x) e^{n^2\pi^2 t}.
\end{align*}
To compute the errors involving $u$ and $\bsigma = -\nabla_\bx u = -\partial_x u$, we compute the series for $n=1,\dots,100$. 
In Figure~\ref{fig:1d_f0_u0hat_par0_pxs2}--\ref{fig:1d_f0_u0hat_par1_pxs2}, we provide the double-logarithmic convergence plots 
on equally scaled and parabolically scaled meshes, respectively. 
For $LS(f,u_0;\bu_h)$ and the component $\|\gradx u_h+\bsigma_h\|_{L^2(Q)}$, we observe the reduced approximate rate $\mathcal{O}(h_\bx^{0.76})=\mathcal{O}({\rm dofs}^{-0.38})$ if $s=1$. 
The errors $\| u - u_h \|_{L^2(Q)}$ and $\| \Pi f -\divst\bu_h\|_{L^2(Q)}=\| f -\divst\bu_h\|_{L^2(Q)}$ converge both at rate $\mathcal{O}(h_\bx^{0,76+1/2})=\mathcal{O}({\rm dofs}^{-0.63})$ for $s=1$ and $\mathcal{O}(h_\bx^2)=\mathcal{O}({\rm dofs}^{-2/3})$ for $s=2$, suggesting that the estimates of Theorem~\ref{thm:L2est} and~\ref{thm:conservation} are indeed sharp. 
For $\| \bsigma - \bsigma_h \|_{L^2(Q)}$, we observe worse rates than~\eqref{eq:a_priori_sigma}, namely, $\mathcal{O}(h_\bx^{0.76}) = \mathcal{O}({\rm dofs}^{-0.38})$ for $s=1$ and $\mathcal{O}(h_\bx^{3/2}) = \mathcal{O}({\rm dofs}^{-1/2})$ for $s=2$.
Proposition~\ref{prop:EndTime} predicts that $\|u(0)-u_h(0)\|_{L^2(\Omega)}$ and $\|u(T)-u_h(T)\|_{L^2(\Omega)}$ converge at rate $\mathcal{O}(h_\bx^{0.76})=\mathcal{O}({\rm dofs}^{-0.38})$ for $s=1$ and $\mathcal{O}(h_\bx^{5/4})=\mathcal{O}({\rm dofs}^{-5/12})$ for $s=2$.
We indeed observe $\|u(0)-u_h(0)\|_{L^2(\Omega)}=\mathcal{O}(h_\bx^{0.76})=\mathcal{O}({\rm dofs}^{-0.38})$ for $s=1$, while we observe $\|u(0)-u_h(0)\|_{L^2(\Omega)}=\mathcal{O}(h_\bx^{3/2})=\mathcal{O}({\rm dofs}^{-1/2})$ for $s=2$. 
In both cases, $\|u(T)-u_h(T)\|_{L^2(\Omega)}$ exhibits the faster rate $\mathcal{O}(h_\bx^2)=\mathcal{O}({\rm dofs}^{-2/(1+s)})$.


\plotMyFigure{1}{f0_u0hat}{0}{2}{	
	\addplot[black,dashed,domain=\firstdof:\lastdof] {2.5*x^(-0.38) };
	\node at (axis cs:1e5,2.5*1.0e-1) [anchor=north west] {$\mathcal{O}({\rm dofs}^{-0.38})$};
	\addplot[black,dashed,domain=\firstdof:\lastdof] {0.09*x^(-0.63) };
	\node at (axis cs:4*1e4,0.9*1.0e-5) [anchor=north west] {$\mathcal{O}({\rm dofs}^{-0.63})$};
	\addplot[black,dashed,domain=\firstdof:\lastdof] {0.002*x^(-1) };
	\node at (axis cs:1e4,1.0e-8) [anchor=north west] {$\mathcal{O}({\rm dofs}^{-1})$};
}


\plotMyFigure{1}{f0_u0hat}{1}{2}{
	\addplot[black,dashed,domain=\firstdof:\lastdof] {x^(-1/3) };
	\node at (axis cs:1e5,1.5*1.0e-1) [anchor=north west] {$\mathcal{O}({\rm dofs}^{-1/3})$};
	\addplot[black,dashed,domain=\firstdof:\lastdof] {0.4*x^(-1/2) };
	\node at (axis cs:4*1e5,2.5*1.0e-4) [anchor=north west] {$\mathcal{O}({\rm dofs}^{-1/2})$};
	\addplot[black,dashed,domain=\firstdof:\lastdof] {0.005*x^(-2/3) };
	\node at (axis cs:1e0,5*1.0e-4) [anchor=north west] {$\mathcal{O}({\rm dofs}^{-2/3})$};
}

\subsection{Smooth solution on unit square}\label{sec:2d_square_sincos}

We consider the solution of the heat equation~\eqref{eq:heat} 
\begin{align*}
  u(t,x_1,x_2) := \cos(\pi t) \sin(\pi x_1) \sin(\pi x_2) \quad \forall (t,x_1,x_2) \in \overline Q = \overline I\times \overline\Omega =[0,1]^3.
\end{align*}
and choose the data $f$ and $u_0$ accordingly. 
In Figure~\ref{fig:2d_square_sincos_par0_pxs2}--\ref{fig:2d_square_sincos_par1_pxs2}, we provide the double-logarithmic convergence plots 
on equally scaled and parabolically scaled meshes, respectively. 
In terms of $h_\bx$, we observe the same convergence rates as in the one-dimensional case of Section~\ref{sec:1d_sincos}. 


\plotMyFigure{2}{square_sincos}{0}{2}{
	\addplot[black,dashed,domain=\firstdof:\lastdof] {20*x^(-1/3) };
	\node at (axis cs:2*1e4,2*1.0e0) [anchor=north west] {$\mathcal{O}({\rm dofs}^{-1/3})$};
	\addplot[black,dashed,domain=\firstdof:\lastdof] {2.5*x^(-2/3) };
	\node at (axis cs:5*1e3,1.0e-3) [anchor=north west] {$\mathcal{O}({\rm dofs}^{-2/3})$};
}


\plotMyFigure{2}{square_sincos}{1}{2}{
	\addplot[black,dashed,domain=\firstdof:\lastdof] {6*x^(-1/4) };
	\node at (axis cs:3*1e4,1.25*1.0e0) [anchor=north west] {$\mathcal{O}({\rm dofs}^{-1/4})$};
	\addplot[black,dashed,domain=\firstdof:\lastdof] {1.5*x^(-1/2) };
	\node at (axis cs:5*1e3,4*1.0e-3) [anchor=north west] {$\mathcal{O}({\rm dofs}^{-1/2})$};
}

\subsection{Non-smooth solution on unit square}\label{sec:2d_square_f0_u0hat}

We choose 
\begin{align*}
  f(t,x_1,x_2) := 0 \quad\forall (t,x_1,x_2)\in\overline Q = \overline I\times \overline \Omega = [0,1]^3
\end{align*}
and
\begin{align*}
  u_0(x_1,x_2) := \big(1-2|x_1-\tfrac12|\big) \big(1-2|x_2-\tfrac12|\big) \quad\forall (x_1,x_2)\in\overline \Omega = [0,1]^2.
\end{align*}
The corresponding solution of the heat equation~\eqref{eq:heat} can be expressed as the Fourier series
\begin{align*}
  u(t,x_1,x_2) = \sum_{n_1=1}^\infty \sum_{n_2=1}^\infty \Big[4 \int_0^1\int_0^1 u_0(y_1,y_2) \sin(n_1\pi y_1) \sin(n_2\pi y_2) \,{\rm d}y_1 \,{\rm d}y_2\Big]\sin(n_1\pi x_1)\sin(n_2\pi x_2) e^{(n_1^2+n_2^2)\pi^2 t}.
\end{align*}
To compute the errors involving $u$ and $\bsigma = -\nabla_\bx u$, we compute the series for $n_1,n_2=1,\dots,100$. 
In Figure~\ref{fig:2d_square_f0_u0hat_par0_pxs2}--\ref{fig:2d_square_f0_u0hat_par1_pxs2}, we provide the double-logarithmic convergence plots 
on equally scaled and parabolically scaled meshes, respectively. 
Once again, in terms of $h_\bx$, we essentially observe the same rates as in the one-dimensional case of Section~\ref{sec:1d_f0_u0hat}. 
Only for $\| u(T) - u_h(T) \|$, no clear rate can be identified in case of $s=1$. 


\plotMyFigure{2}{square_f0_u0hat}{0}{2}{
	\addplot[black,dashed,domain=\firstdof:\lastdof] {2*x^(-0.25) };
  \node at (axis cs:3*1e4,1.5*1.0e0) [anchor=north west] {$\mathcal{O}({\rm dofs}^{-1/4})$};
	\addplot[black,dashed,domain=\firstdof:\lastdof] {0.25*x^(-0.5) };
  \node at (axis cs:5*1e3,7*1.0e-4) [anchor=north west] {$\mathcal{O}({\rm dofs}^{-1/2})$};
}


\plotMyFigure{2}{square_f0_u0hat}{1}{2}{
	\addplot[black,dashed,domain=\firstdof:\lastdof] {1.5*x^(-1/4) };
	\node at (axis cs:1e5,1.5*1.0e0) [anchor=north west] {$\mathcal{O}({\rm dofs}^{-1/4})$};
	\addplot[black,dashed,domain=\firstdof:\lastdof] {0.5*x^(-3/8) };
	\node at (axis cs:7.5*1e5,3*1.0e-3) [anchor=north west] {$\mathcal{O}({\rm dofs}^{-3/8})$};
	\addplot[black,dashed,domain=\firstdof:\lastdof] {0.1*x^(-1/2) };
	\node at (axis cs:1e4,3*1.0e-4) [anchor=north west] {$\mathcal{O}({\rm dofs}^{-1/2})$};
	\addplot[black,dashed,domain=\firstdof:\lastdof] {1e-7*x^(-1/2) };
	\node at (axis cs:1e4,1.0e-8) [anchor=north west] {$\mathcal{O}({\rm dofs}^{-1/2})$};
}

\bibliographystyle{alpha}
\bibliography{literature}

\end{document}